\documentclass[english,12pt]{smfart}
  \usepackage{amsmath,amssymb,smfthm}
  \usepackage{mathrsfs,mathtools}
  \usepackage{graphicx,color}

\author{Beno\^{\i}t Kloeckner}
\title[Circle expanding maps]{Optimal transport and dynamics of 
  expanding circle maps acting on measures}

\newcommand{\measures}{\mathscr{P}(\mathbb{S}^1)}
\newcommand{\Id}{\mathrm{Id}}
\newcommand{\wassd}{\mathop \mathrm{W}\nolimits}
\newcommand{\mdim}{\mathop \mathrm{mdim}_M\nolimits}
\newcommand{\dd}{\mathrm{d}}
\newcommand{\smallsum}{{\textstyle \sum}}
\renewcommand{\smallint}{{\textstyle\int}}

\newtheorem{ques}{Question}

\begin{document}
%%%%%%%%%%%%%%%%%%%%%%%%%%%%%%%%%%%%%%%%%%%%%%%%%%%%%%%%%%%%%%%%%%%%%%%%%%%%%%%%
%%%%%%%%%%%%%%%%%%%%%%%%%%%%%%%%%%%%%%%%%%%%%%%%%%%%%%%%%%%%%%%%%%%%%%%%%%%%%%%%
%%%%%%%%%%%%%%%%%%%%%%%%%%%%%%%%%%%%%%%%%%%%%%%%%%%%%%%%%%%%%%%%%%%%%%%%%%%%%%%%

\begin{abstract}
In this article we compute the derivative of the action on probability measures
of an expanding circle map at its absolutely continuous invariant measure.
The derivative is defined using optimal transport: we use the rigorous framework
set up by N. Gigli to endow the space of measures with a kind of differential structure.

It turns out that $1$ is an eigenvalue of infinite
multiplicity of this derivative, and we deduce that the absolutely continuous invariant measure can be deformed in many ways into atomless, nearly invariant measures. As a consequence, we obtain
counter-examples to an infinitesimal version of Furstenberg's conjecture.

We also show that the action of standard self-covering maps on measures
has positive metric mean dimension. 
\end{abstract}

\thanks{This work was supported by the Agence Nationale de la Recherche, grant ANR-11-JS01-0011.}

\maketitle

Some time after the publication of a first version of this article
\cite{Original}, I found an application of the results to an infinitesimal 
version of Furstenberg's conjecture, as well as an error
in a Lemma (which could be corrected without affecting the main 
results). Both are to be published together in \textit{Ergodic Theory
and Dynamical system}, and the present article is a consolidated
version which combines the original article with the error corrected,
and the additional material.

%%%%%%%%%%%%%%%%%%%%%%%%%%%%%%%%%%%%%%%%%%%%%%%%%%%%%%%%%%%%%%%%%%%%%%%%%%%%%%
%%%%%%%%%%%%%%%%%%%%%%%%%%%%%%%%%%%%%%%%%%%%%%%%%%%%%%%%%%%%%%%%%%%%%%%%%%%%%%
\section{Introduction}

The theory of optimal transport has drawn much attention in recent years.
Its applications to geometry and PDEs have in particular been largely 
disseminated. In this paper, we would like to show its effectiveness in a 
dynamical context. We are interested in arguably the simplest dynamical
system where the action on measures is significantly different from the action 
on points, namely expanding circle maps.

Another goal of the paper is to examplify the rigorous differential structure
defined by N. Gigli \cite{Gigli}, for the simplest possible compact manifold.
Note that one can use absolutely continuous curves to define the almost 
everywhere differentiability of maps, see in particular
\cite{Gigli2} where this method is applied to the exponential map.
Other previous uses of variants of this manifold structure
include the definition of gradient flows, as in the pioneering \cite{Otto} and in 
\cite{Ambrosio-Gigli-Savare}, and of curvature, as in \cite{Lott}.
But to our knowledge, no example of explicit derivative of a measure-defined
map at a given point had been computed before.

%%%%%%%%%%%%%%%%%%%%%%%%%%%%%%%%%%%%%%%%%%%%%%%%%%%%%%%%%%%%%%%%%%%%%%%%%%%%%%%%%
\subsection{An important model example}

Let us first consider the usual degree $d$
self-covering map of the circle $\mathbb{S}^1 = \mathbb{R}/\mathbb{Z}$ defined by
\[\Phi_d(x) = dx \mod 1.\]
It acts on the set $\measures$ of Borel probability measures,
endowed with the topology of weak convergence, by the push-forward map
$\Phi_{d\#}$. 

A map like $\Phi_d$ can act by composition on the right on a
function space (e.g. Sobolev spaces). 
The adjoint of this map is usually called
a Perron-Frobenius operator or a transfer operator, and a great 
deal of effort has been made to understand these
operators, especially their spectral properties (see for example \cite{Baladi}). 
One can consider
$\Phi_{d\#}$ as an analogue for possibly singular measures of the Perron-Frobenius
operator of $\Phi_d$.

As pointed out by the referee of a previous version of this paper, using the
finite-to-one maps 
\[(x_1,\ldots,x_n)\mapsto \frac1n\delta_{x_1}+\dots+\frac1n\delta_{x_n}\]
it is easy to prove that $\Phi_{d\#}$ is topologically transitive and has infinite
topological entropy. To refine this last remark, we shall prove that $\Phi_{d\#}$ 
has positive metric mean dimension (a metric dynamical invariant of infinite-entropy
maps).
\begin{theo}\label{theo:mdim}
For all integers $d\geqslant 2$ and all exponents $p\in[1,+\infty)$ we have 
\[\mdim(\Phi_{d\#},\wassd_p)\geqslant p(d-1)\]
where $\wassd_p$ is the Wasserstein metric with cost $|\cdot|^p$.
\end{theo}
The definition of Wasserstein metrics is given below; for the definiton
of metric mean dimension and the proof of the above result, see 
Section \ref{sec:entropy}.
Except in this result, we shall only use the quadratic Wasserstein metric
($p=2$), which we will denote by $\wassd$.

Our main goal is to study the first-order dynamics of $\Phi_{d\#}$ near the uniform
measure $\lambda$. The precise setting will be exposed latter; let us just give
a few elements. The tangent space $T_\mu$ to $\measures$ at a measure
$\mu$ that is absolutely continuous with a bounded and bounded away from zero
density identifies
with the Hilbert space $L^2_0(\mu)$ of all vector fields 
$v:\mathbb{S}^1\to\mathbb{R}$ that are $L^2$ with respect to $\mu$,
and such that $\int v \,\lambda =0$. More generally,
if $\mu$ is atomless $T_\mu$ identifies with a Hilbert subspace
$L^2_0(\mu)$ of $L^2(\mu)$.

We have a kind of exponential map: $\exp_\mu(v)=\mu+v:=(\Id+v)_\#\mu$. 
Then we say that a map $f$ acting on $\measures$
has G\^ateau derivative $L$ at $\mu$ if $f(\mu)$ has no atom
and $L:L^2_0(\mu)\to L^2_0(f(\mu))$ is a continous linear operator
such that for all $v$ we have
\[\wassd(f(\mu+tv),f(\mu)+tLv)=o(t)\]
We then write $D_\mu f = L$.

Our first differentiability result is the following.
\begin{theo}\label{theo:diff}
The map $\Phi_{d\#}$ has a G\^ateaux derivative at $\lambda$,
equal to $d$ times
the Perron-Frobenius operator of $\Phi_d$ acting on $L^2_0(\lambda)$.
In particular its spectrum is the disc of radius $d$ and all numbers of modulus $<d$ are
eigenvalues with infinite multiplicity.
\end{theo}
This result is detailled as Theorem \ref{theo:differential} and Proposition \ref{prop:spec}
below. We shall also see that $\Phi_{d\#}$ is not Fr\'echet differentiable.

%%%%%%%%%%%%%%%%%%%%%%%%%%%%%%%%%%%%%%%%%%%%%%%%%%%%%%%%%%%%%%%%%%%%%%%%%%%%%%
\subsection{General expanding maps}

The next step is to consider the action on measures of expanding
circle maps. In Section \ref{sec:general}, given a general $C^2$
expanding map $\Phi$, we compute the derivative of $\Phi_\#$ at 
its unique absolutely continuous invariant measure (Theorem \ref{theo:expanding}). 
Instead of writting down the expression here, let us simply state the following.
\begin{theo}\label{theo:generaldiff}
If $\Phi$ is a $C^2$ expanding circle map, $\Phi_\#$ has a G\^ateaux derivative at its
unique invariant absolutely continuous measure $\rho\lambda$, whose adjoint operator 
in  $L^2_0(\rho\lambda)$ is $u\mapsto \Phi'\,u\circ\Phi$.
\end{theo}
In particular this derivative is a multiple of the Perron-Fronenius
operator (on $L^2_0(\rho\lambda)$) only when $\Phi'$ is constant, 
that is when $\Phi$ is a model map. Using general results
in the spectral theory of transfer operator, it is however possible to
prove that $1$ is always an eigenvalue of infinite multiplicity,
with continuous eigenfunctions.

%%%%%%%%%%%%%%%%%%%%%%%%%%%%%%%%%%%%%%%%%%%%%%%%%%%%%%%%%%%%%%%%%%%%%%%%%%%%%%
\subsection{Nearly invariant measures}

The spectral
study of $D_\lambda(\Phi_\#)$ gives us large families of nearly invariant
measures, with Lipschitz para\-metrization.
\begin{theo}\label{theo:almost-invariant}
For all integers $n$, there is a bi-Lipschitz embedding 
$F :  B^n \to \measures$ mapping $0$ to the absolutely continuous
invariant measure $\rho\lambda$ of $\Phi$ such that
\[\wassd\big(\Phi_\#(F(a)),F(a)\big) =o(|a|).\]
As a consequence, for all $\varepsilon>0$ and all integer $K$
there is a radius $r>0$ such that for all $k\leqslant K$ and
all $a\in B^n(0,r)$ the following holds:
\[\wassd\big(\Phi_{d\#}^k(F(a)),F(a)\big)\leqslant \varepsilon |a|.\]
\end{theo}
Here $B^n$ denotes the unit Euclidean ball centered at $0$ and $\wassd$ is the quadratic
Wasserstein distance (whose definition is recalled below).

It is easy to construct invariant measures near the 
absolutely continuous one, for example supported on a union of periodic orbits.
One can also
consider convex sums $(1-a)\rho\lambda+a\mu$ where $\mu$ is any invariant measure
and $a\ll 1$. But note that the curves $a\mapsto (1-a)\rho\lambda+a\mu$ need not
be rectifiable, let alone Lipschitz. Bernoulli measures are also examples;
they are singular, atomless, fully supported invariant measures of $\Phi_d$
that can be arbitrary close to $\lambda$. 

The nearly invariant measures above seem of a different
nature, and a natural question is how regular they are.
They are given by push-forwards of the uniform measure by continuous 
functions; 
for example in the model case a one parameter family is given by
\[\big(\Id+t\sum_{\ell=0}^\infty d^{-\ell}\cos(2\pi d^\ell \cdot)\big)_\#\lambda   \]
where $t\in[0,\varepsilon)$. This makes it easy to prove that almost all
of them are atomless.
\begin{prop}\label{prop:atomless}
If $\mu$ is an atomless measure and $v\in L^2(\mu)$,
 for all but a countable number of values of $t\in[0,1]$, the
measure $\mu+tv=(\Id+tv)_{\#}\mu$ has no atom.

In particular, with the notation of Theorem \ref{theo:almost-invariant},
the map $F$ can be chosen such that
for almost all $a$ the measure $F(a)$ has no atom.
\end{prop}
That the first part of this result implies the second part shall
become clear during the proof of Theorem \ref{theo:almost-invariant}
in Section \ref{sec:nearly-invariant}, where we construct $F$.

This leaves open the following, antagonist questions.
\begin{ques}
Is the measure $F(a)$ absolutely continuous for most, or at least some $a\neq 0$?
\end{ques}

\begin{ques}
Is the measure $F(a)$ invariant for most, or at least some $a\neq 0$?
\end{ques}

The next natural questions, not adressed at all here, concerns the 
dynamical properties of the action
on measures of higher dimensional hyperbolic dynamical systems like Anosov
maps or flows, or of discontinuous systems like interval exchange maps.

%%%%%%%%%%%%%%%%%%%%%%%%%%%%%%%%%%%%%%%%%%%%%%%%%%%%%%%%%%%%%%%%%%%%%%%%%%%%%%%%
\subsection{An infinitesimal version of Furstenberg's conjecture}

While circle expanding maps have in many respects become toy-models
in the category of hyperbolic dynamical systems, a prominent question
concerning them is still open for over half a century.
\begin{conj}[Furstenberg]
If an atomless probability measure $\mu$ on the circle $\mathbb{S}^1=\mathbb{R}/\mathbb{Z}$
is invariant under both
\[\Phi_2:x\mapsto 2x\mod 1 \qquad\mbox{and}\qquad \Phi_3:x\mapsto 3x\mod 1 \]
then $\mu$ is equal to the Lebesgue measure $\lambda$.
\end{conj}
In the above conjecture, one can replace $2$ and $3$ by two multiplicatively
independent integers. Even the above case is wide open in general, though
a theorem of Rudolph asserts that Furstenberg's
conjecture holds for measures $\mu$ having positive entropy for one of the maps
$\Phi_2$ or $\Phi_3$ \cite{Rudolph} (see also \cite{Johnson}). Many other results related to this question can be found in the literature, among which \cite{Hochman-Shmerkin,BLMV}; the 
interested reader can for example use the answers to the 
MathOverflow question \cite{MO} as pointers.

To see how relevant our results
can be in the context of Furstenberg conjecture, let us consider how one
can approach this kind of problem in a differential geometric setting.

Furstenberg's conjecture is a strong rigidity statement; in differential geometry,
a common strategy to attack such questions is to aim to weaker rigidity statements.
A first weakening would be to ask whether the point known to have a given property of 
interest  (here: $\lambda$) is, rather than unique, at least isolated among points with 
this property? If this stays out of reach, then can we prove that it is not possible
to deform this point, i.e. to find a non-constant continuous path starting at this point
inside the set defined by the given property? A further weakening is to ask for
first-order rigidity, i.e. to ask whether we can use the tangent space and derivatives
to prove that no $C^1$ deformation can exist in the considered set. In the case of Furstenberg
conjecture, we have a space $\mathcal{P}(\mathbb{S}^1)$ and two rather rich subspaces,
the sets of atomless invariant measures for $\Phi_2$ and $\Phi_3$. Let us denote these 
sets of fixed measures
by $I_2$ and $I_3$; then the conjecture is that
$I_2\cap I_3=\lambda$.
 Imagine for a moment that $I_2$ and $I_3$ are some sort of differentiable submanifolds of
$\mathcal{P}(\mathbb{S}^1)$; then the various above weakenings of Furstenberg's conjecture
would take the form of the following questions: 
\begin{enumerate}
\item Is $\lambda$ isolated in
$I_2\cap I_3$? 
\item Is $\lambda$ the sole point in its path-connected component inside
$I_2\cap I_3$?
\item Must a $C^1$ curve starting at $\lambda$ and lying inside
  $I_2\cap I_3$ be constant?
\end{enumerate}

Finally, to prove a positive answer to this third weakening,
the most common approach would be to prove that the
intersection $I_2\cap I_3$ is ``first-order rigid''
at $\lambda$, in the sense
that the tangent spaces $T_\lambda I_2$ and $T_\lambda I_3$ intersect trivially.

Since $I_2$ and $I_3$ are defined (if we forget momentarily the atomless condition)
as sets of fixed points fo $\Phi_{2\#}$ and $\Phi_{3\#}$, the first-order rigidity question
would reduce to ask whether the spaces $E_2,E_3 \subset T_\lambda \mathcal{P}(\mathbb{S}^1)$
of invariant vectors for the derivatives
$D_\lambda\Phi_{2\#}$ and $D_\lambda\Phi_{3\#}$ intersect trivially. Even if all the above 
speculation turns out to be wrong (e.g. $I_2$ and $I_3$ could not be 
anything close to submanifolds), this last question is perfectly defined in the
differential setting alluded to above, and can be considered an infinitesimal
version of Furstenberg's conjecture. 
As a consequence of the previous results, we will prove that this
question as a negative answer.

\begin{theo}\label{theo:main}
The vector space $E_2\cap E_3\subset T_\lambda \mathcal{P}(\mathbb{S}^1)$
of tangent vectors at $\lambda$ that are simultaneously invariant under both
$D_\lambda \Phi_{2\#}$ and $D_\lambda \Phi_{3\#}$ is infinite-dimensional.

The vector space $\bigcap_{d=2}^\infty E_d$ of tangent vectors
at $\lambda$ that are simultaneously invariant under all
the $D_\lambda \Phi_{d\#}$ is $2$-dimensional.
\end{theo}

Formulated as it is in terms of the Wasserstein metric,
this result could feel very abstract and potentially artificial, so let us give a
direct corollary that contains no reference to optimal transport or abstract
differential geometric setting. The idea behind this corollary goes back
to an insight of Otto \cite{Otto} related to the point of
view of Benamou and Brenier \cite{Benamou-Brenier} and 
developed
in \cite{Ambrosio-Gigli-Savare}:
by integration, smooth test functions can serve as a kind of (weak) coordinates on
$\mathcal{P}(\mathbb{S}^1)$; for simplicity this corollary is phrased 
in a restricted way, only using that  $\bigcap_{d=2}^\infty E_d$ is not reduced to $0$.

\begin{coro}\label{coro:main}
There exists a path of probability measures
 $(\mu_t)_{t\in(-\varepsilon,\varepsilon)}$
with $\mu_0=\lambda$, continuous in the weak topology, 
with $\mu_t$ atomless for almost all $t$, such that:
\[\frac{\dd}{\dd t} \int_{\mathbb{S}^1} \psi_0 \,\dd\mu_t \Big\rvert_{t=0} \neq 0\]
for some smooth function $\psi_0:\mathbb{S}^1\to\mathbb{R}$, and
\[\frac{\dd}{\dd t}\int_{\mathbb{S}^1} \psi \,\dd\mu_t \Big\rvert_{t=0}
  = \frac{\dd}{\dd t}\int_{\mathbb{S}^1} \psi \,\dd \big(\Phi_{d\#}\mu_t\big) 
     \Big\rvert_{t=0}\]
for all smooth functions $\psi:\mathbb{S}^1\to\mathbb{R}$ and all integer $d\ge2$.
\end{coro}

\begin{rema}\begin{enumerate}
\item The first condition ensures that $\mu_t$ depends significantly on $t$ (in particular,
it avoids the degenerate and obvious choice $\mu_t\equiv\lambda$), while
the second condition expresses that for small $t$, $\mu_t$ is ``almost invariant''
under all the push-forward maps $\Phi_{d\#}$. Of course, this condition can be rewritten
\[\frac{\dd}{\dd t}\int_{\mathbb{S}^1} \psi \,\dd\mu_t \Big\rvert_{t=0}
  = \frac{\dd}{\dd t}\int_{\mathbb{S}^1} \psi\circ\Phi_d \,\dd \mu_t \Big\rvert_{t=0}.\]
\item This corollary is intrinsically much weaker than the theorem, as 
  differentiability in the sense of Wasserstein distance implies
  differentiability of the integrals of test functions, but the converse
  implication does not hold. For example, a curve of the form
  $(t\mu + (1-t)\nu)_t$ is usually not differentiable (or even rectifiable)
  in the differential structure induced by $\wassd_2$, while
  the integral of any test function depends affinely on $t$. Nevertheless,
  I do not know a simpler way to get Corollary \ref{coro:main} even when
  restricting $d$ to $\{2,3\}$. Even if the
  Furstenberg conjecture where false and there where an atomless probability measure
  $\mu\neq \lambda$ invariant by $\Phi_2$ and $\Phi_3$, the curve
  $(t\lambda+(1-t)\mu)_t$ would not work as these measures are not positive
  for negative $t$.
\item One could try to extend this infinitesimal argument to the construction
  of families of counter-examples to the Furstenberg conjecture: if
  one of the invariant vectors we found could be extended to a
  vector field preserved by both $\Phi_2$ and $\Phi_3$,
  then the integral curve issued from $\lambda$ would be entirely
  made of invariant measures for both $\Phi_2$ and $\Phi_3$.
  However, it would be incredibly bold to conjecture this extension
  to be possible: we do not even know whether $\Phi_{2\#}$ is
  differentiable at any non-absolutely continuous measure.
  Note also
  that this extension cannot be expected at all for the full semi-group 
  $\mathbb{N}$, as it is known that
  the Lebesgue measure is the only atomless measure invariant under all
  $\Phi_d$ (this holds more generally for large enough sub-semigroups of $\mathbb{N}$,
  see \cite{Einsiedler-Fish}). One can still dream
  of making this approach work for 
  finitely generated multiplicative sub-semigroups, as this
  case is very different from larger sub-semigroups:
  in the former case, the remainder in
  the first-order Taylor formula for the $\Phi_{d\#}$ at 
  $\lambda$ can be made 
  uniform over the generators (for a fixed simultaneously 
  invariant tangent vector).
\end{enumerate}\end{rema}

%%%%%%%%%%%%%%%%%%%%%%%%%%%%%%%%%%%%%%%%%%%%%%%%%%%%%%%%%%%%%%%%%%%%%%%%%%%%%%%%%%%%%%%%%%
\subsection{Recalls and notations}

The most convenient point of view here is to construct the circle as
the quotient $\mathbb{R}/\mathbb{Z}$. We shall often and without notice write
a real number $x\in[0,1)$ to mean its image by the canonical projection. We proceed
similarly for intervals of length less than $1$.

Recall that the push-forward of a measure is defined by 
$\Phi_\#\mu(A)=\mu(\Phi^{-1}A)$ for all Borel sets $A$.

For a detailled introduction on optimal transport, the interested reader can for
example consult \cite{Villani}. Let us give an overview of the properties we shall need.
Given an exponent $p\in[1,\infty)$, if $(X,d)$ is a general metric space, assumed to be polish (complete 
separable) to avoid mesurability issues and endowed with its Borel 
$\sigma$-algebra, its $L^p$ \emph{Wasserstein space}  is
the set $\mathscr{W}_p(X)$ of probability measures $\mu$ on $X$ whose $p$-th moment is finite:
\[\int d^p(x_0,x) \,\mu(dx)<\infty\qquad\mbox{ for some, hence all }x_0\in X\]
endowed with the following metric: given $\mu,\nu\in\mathscr{W}_p(X)$ one sets
\[\wassd_p(\mu,\nu)=\left(\inf_\Pi \int_{X\times X} d^p(x,y)\, 
  \Pi(dx dy)\right)^{1/p}\]
where the infimum is over all probability measures $\Pi$ on $X\times X$
that projects to $\mu$ on the first factor and to $\nu$ on the second one.
Such a measure is called a transport plan between $\mu$ and $\nu$, and is
said to be optimal when it achieves the infimum. In this setting, an optimal
transport plan always exists. Note that when $X$ is compact, the set $\mathscr{W}_p(X)$
is equal to the set $\mathscr{P}(X)$ of all probability measures on $X$.

The name ``transport plan'' is suggestive: it is a way to describe what amount of
mass is transported from one region to another.

The function $\wassd_p$ is a metric, called the ($L^p$) Wasserstein metric, 
and when $X$ is compact it induces the weak topology. We sometimes
denote $\wassd_2$ simply by $\wassd$.

%%%%%%%%%%%%%%%%%%%%%%%%%%%%%%%%%%%%%%%%%%%%%%%%%%%%%%%%%%%%%%%%%%%%%%%%%%%%%%%%
%%%%%%%%%%%%%%%%%%%%%%%%%%%%%%%%%%%%%%%%%%%%%%%%%%%%%%%%%%%%%%%%%%%%%%%%%%%%%%%%
%%%%%%%%%%%%%%%%%%%%%%%%%%%%%%%%%%%%%%%%%%%%%%%%%%%%%%%%%%%%%%%%%%%%%%%%%%%%%%%%
\section{Metric mean dimension}\label{sec:entropy}

Metric mean dimension is a metric invariant of dynamical systems introduced by
Lindenstrauss and Weiss \cite{Lindenstrauss-Weiss}, that refines topological entropy
for infinite-entropy systems.

Let us briefly recall the definitions. Given a
map $f:X\to X$ acting on a compact metric space, for any
$n\in\mathbb{N}$ one defines a new metric on $X$ by
\[d_n(x,y):= \max\{d(f^k(x),f^k(y));0\leqslant k\leqslant n\}.\]
Given $\varepsilon>0$, one says that a subset $S$ of $X$ is
$(n,\varepsilon)$-separated if $d_n(x,y)\geqslant \varepsilon$ whenever
$x\neq y\in S$. Denoting by $N(f,\varepsilon,n)$ the maximal size of a 
$(n,\varepsilon)$-separated set, the topological entropy of $f$ is defined as
\[h(f) := \lim_{\varepsilon\to 0} \limsup_{n\to+\infty} 
\frac{\log N(f,\varepsilon,n)}{n}.\]
Note that this limit exists since $\limsup_{n\to+\infty} \frac1n \log N(f,\varepsilon,n)$
is nonincreasing in $\varepsilon$.
The adjective ``topological'' is relevant since $h(f)$ does not depend upon the
distance on $X$, but only on the topology it defines.
The topological entropy is in some sense a global measure of the dependance on initial condition
of the considered dynamical system. 
The map $\Phi_d$ is a classical example, whose topological entropy is $\log d$.

Now, the metric mean dimension is
\[\mdim(f,d) := \liminf_{\varepsilon\to 0} \limsup_{n\to+\infty} 
  \frac{\log N(f,\varepsilon,n)}{n|\log\varepsilon|}.\]
It is zero as soon as topological entropy is finite. Note that this quantity
does depend upon the metric; here we shall use $\wassd_p$.
Lindenstrauss and Weiss define the metric mean dimension using
covering sets rather than separated sets, but this does not matter since
their sizes are comparable.

Let us prove Theorem \ref{theo:mdim}:
the metric mean dimension of $\Phi_{d\#}$ is at least $p(d-1)$ when
$\measures$ is endowed with the $W_p$ metric.
In another paper \cite{Kloeckner2}, we prove the same kind of result,
replacing $\Phi_d$ by any map having positive entropy. However
Theorem \ref{theo:mdim} has a better constant and its proof is simpler.

\begin{proof}[Proof of Theorem \ref{theo:mdim}]
To construct
a large $(n,\varepsilon)$-separated set, we proceed as follows: we start with the point
$\delta_0$, and choose an $\varepsilon$-separated set of its antecedents. Then we inductively
choose $\varepsilon$-separated sets of antecedents of each elements of the set 
previously defined.
Doing this, we need not control the distance between antecedents of two different elements.
 
Let $k\gg 1$ and $\alpha>0$ be integers; $\varepsilon$ will be exponential in $-k$. Let
$A_k$ be the set all $\mu\in\measures$ such that $\mu((1-2^{-k},1))=0$
and $\mu([0,1/d])\geqslant 1/2$. These conditions are designed to bound from
below the distances between the antecedents to be constructed: a given amount 
of mass (second condition) will have to travel a given distance (first
condition).

An element $\mu\in A_k$ decomposes as $\mu=\mu_h+\mu_t$ where
$\mu_h$ is supported on $[0,1-d2^{-k}]$ and $\mu_t$ is supported
on $(1-d2^{-k},1-2^{-k})$. Let $e_1,\ldots, e_d$ be the right inverses to
$\Phi$ defined onto $[0,1/d), [1/d,2/d),\ldots [(d-1)/d,1)$ respectively.
For all integer tuples $\ell=(\ell_1,\ldots,\ell_d)$ such that $\ell_1\geqslant 2^{\alpha k-1}$
and $\sum \ell_i=2^{\alpha k}$, define
\[\mu_\ell=e_{1\#}(\ell_1 2^{-\alpha k}\mu_h+\mu_t)+\sum_{i>1} e_{i\#}(\ell_i 2^{-\alpha k}\mu_h)\]
(see figure \ref{fig:antecedents} that illustrates the case $d=2$).
It is a probability measure on $\mathbb{S}^1$,
lies in $A_k$ and $\Phi_{d\#}(\mu_\ell)=\mu$. Moreover, if $\ell'\neq\ell$
then the masses given by $\mu_{\ell'}$ and $\mu_\ell$ to some interval
$[(i-1)/d,i/d)$ differ by at least $2^{-\alpha k-1}$, and
the $2^{-k}d^{-1}$-long rightmost part of each such interval is given
no mass by $\mu_\ell$ and $\mu_{\ell'}$.
It follows that any transport plan $\Pi$ from $\mu_\ell$ to $\mu_{\ell'}$ 
has to give a mass at least $2^{-\alpha k-1}$ to the set
of pairs $(x,y)\in\mathbb{S}^1$ such that $|x-y|\geqslant2^{-k}d^{-1}$
(one sometimes says that $\Pi$ moves a mass at least 
$2^{-\alpha k-1}$ by a distance at least $2^{-k}d^{-1}$).
Therefore,
\[\wassd_p(\mu_\ell,\mu_{\ell'})\geqslant d^{-1}2^{-k(\alpha/p+1)-1/p}.\]

\begin{figure}[htp]\begin{center}
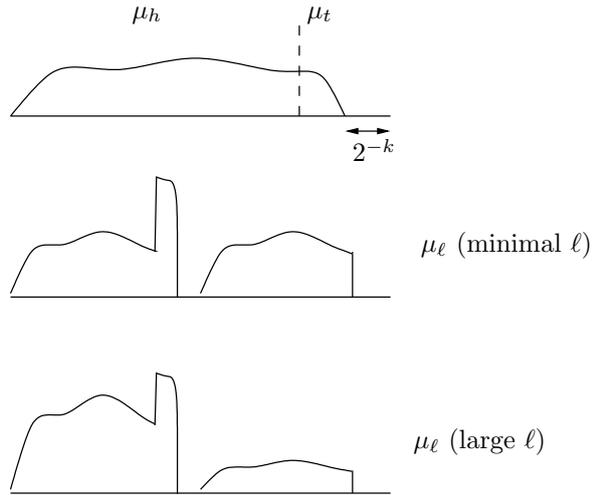
\caption{Construction of separated antecedents of a given measure.}
\label{fig:antecedents}
\end{center}\end{figure}

Let $\varepsilon=d^{-1}2^{-k(\alpha/p+1)-1/p}$
and define $S_n$ inductively as follows.
First, $S_0=\{\delta_0\}$. Given $S_n\subset A_k$, $S_{n+1}$
is the set of all $\mu_\ell$ constructed above, where $\mu$ runs through
$S_n$.

By construction, $S_{n+1}$ has at least $C2^{\alpha k(d-1)}$ times
has many elements as $S_n$, for some constant $C$ depending only on $d$. 
Then $S_n$ has at least  $C^n 2^{n\alpha k(d-1)}$ elements.
Let $\mu$, $\nu$
be two distinct elements of $S_n$ and $m$ be the greatest index such that
$\Phi_{d\#}^m\mu\neq \Phi_{d\#}^m\nu$. Since $\Phi_{d\#}^n\mu=\delta_0=\Phi_{d\#}^n\nu$,
$m$ exists and is at most $n-1$. The measures $\mu'=\Phi_{d\#}^m\mu$ and 
$\nu'=\Phi_{d\#}^m\nu$ both lie in $S_{n-m}$ and have the same image. Therefore,
they are $\varepsilon$-separated. This shows that $S_n$ is $(n,\varepsilon)$-separated.

It follows that 
\begin{eqnarray*}
\frac{\log N(\Phi_{d\#},\varepsilon,n)}{n|\log\varepsilon|}
  &\geqslant& 
  \frac{C}{|\log\varepsilon|}+\frac{\alpha(d-1)}{\frac{\alpha}{p}+1}
  \left(\frac{-\frac1p-\frac{\log d}{\log 2}}{|\log\varepsilon|}+1 \right) \\
  &\geqslant& \frac{\alpha(d-1)}{\frac{\alpha}{p}+1}(1+o(1))+o(1).
\end{eqnarray*}
In the case of a general $\varepsilon$, we get the same bound on
$\log N$ up to an additive term $n\alpha(d-1)\log 2$, so that
\[\mdim(\Phi_{d\#},\wassd_p) \geqslant \frac{\alpha(d-1)}{\frac{\alpha}{p}+1}.\]
By taking $\alpha\to\infty$ we get $\mdim(\Phi_{d\#},\wassd_p)\geqslant p(d-1)$.
\end{proof}

%%%%%%%%%%%%%%%%%%%%%%%%%%%%%%%%%%%%%%%%%%%%%%%%%%%%%%%%%%%%%%%%%%%%%%%%%%%%%%%%
%%%%%%%%%%%%%%%%%%%%%%%%%%%%%%%%%%%%%%%%%%%%%%%%%%%%%%%%%%%%%%%%%%%%%%%%%%%%%%%%
%%%%%%%%%%%%%%%%%%%%%%%%%%%%%%%%%%%%%%%%%%%%%%%%%%%%%%%%%%%%%%%%%%%%%%%%%%%%%%%%
\section{The first-order differential structure on measures}

In this section we give a short account on the work of Gigli \cite{Gigli}
in the particular case of the circle.
Note that considering the Wasserstein space of a Riemannian manifold as an
infinite-dimensionnal Riemannian manifold dates back to the work
of Otto \cite{Otto}. 
However, in many ways it stayed at the formal level until the work of Gigli.

%%%
\subsection{Why bother with this setting?}

Before getting started, let us explain why we do not simply use the natural affine
structure on $\measures$,
the tangent space at a point simply consisting on signed measures having
zero total mass. Similarly, one could consider simpler to just
take the smooth functions of $\mathbb{S}^1$ as coordinates to define a smooth structure
on $\measures$. 

The first argument against these points of view is that optimal transportation is
about pushing mass, not (directly) about recording the variation of density at each point.

More important, these simple ideas would lead to a path of the form 
$\gamma_t=t\delta_x+(1-t)\delta_y$ being smooth. However, the Wasserstein
distance between $\gamma_t$ and $\gamma_s$ has the order of $\sqrt{|t-s|}$,
so that $\gamma_t$ is not rectifiable (it has infinite length)! This also holds,
for example, for convex sums of measures with different supports.

One could argue that the previous paths can be made Lipschitz by using $\wassd_1$
instead of $\wassd_2$, so let us give another argument:
in the affine structure, the Lebesgue measure does not have a tangent space but only a 
tangent cone since $\lambda+t\mu$ is not a positive measure for all small
$t$ unless $\mu\ll\lambda$. If one wants to consider singular measures in the same
setting than regular ones, the $\wassd_2$ setting seems to be the right tool.

Note that it will appear that the differential structure on $\measures$ depends
not only on the differential structure of the circle, but also on its metric.
This should not be considered surprising: in finite dimension, the fact that the
differential structures are defined independently of any reference to a metric comes
from the equivalence of norms in Euclidean space: here, in infinite dimension, even the simple
formula $\wassd(f(\mu+tv),f(\mu)+tD_xf(v)) = o(t)$ involves a metric in a crucial way.

One could also be
surprised that this differential structure involving the metric of the circle could
be preserved by expanding maps of non-constant derivative. This point shall be
cleared in Section \ref{sec:general}, see Proposition \ref{prop:centering} and the
discussion before it.

%%%
\subsection{The exponential map}

 Note that as is customary
in these topics, by a geodesic we mean a non-constant globally minimizing geodesic segment
or line, parametrized proportionaly to arc length.

Given $\mu\in\measures$, there are several equivalent ways to define its
tangent space $T_\mu$. In fact, $T_\mu$ has a vectorial structure only when 
$\mu$ is atomless; otherwise it is only a tangent cone. Note that the atomless
condition has to be replaced by a more intricate one in higher dimension.

The most Riemannian way to construct $T_\mu$ is to use the exponential map.
Let $\mathscr{P}(T\mathbb{S}^1)_\mu$ be the set of probability measures 
on the tangent bundle
$T\mathbb{S}^1$ that are mapped  to $\mu$ by the canonical projection.

Given $\xi,\zeta\in \mathscr{P}(T\mathbb{S}^1)_\mu$, one defines
\[\wassd_\mu(\xi,\zeta) = \left(\inf_\Pi \int_{T\mathbb{S}^1\times T\mathbb{S}^1} d^2(x,y)
  \,\Pi(dx dy)\right)^{1/2}\]
where $d$ is any metric whose restriction to the fibers is the riemannian
distance (here the fibers are isometric to $\mathbb{R}$), and the infimum 
is over transport plans $\Pi$ that are mapped to the identity
$(\Id,\Id)_\#\mu$ by the canonical projection on $\mathbb{S}^1\times 
\mathbb{S}^1$. This means that we allow only to move the mass \emph{along}
the fibers. Equivalently, one can disintegrate $\xi$ and $\zeta$ along $\mu$,
writing $\xi=\int\xi_x \,\mu(dx)$ and $\zeta=\int \zeta_x \,\mu(dx)$, with
$(\xi_x)_{x\in\mathbb{S}^1}$ and $(\zeta_x)_{x\in\mathbb{S}^1}$ two families
of probability measures on $T_x\mathbb{S}^1\simeq \mathbb{R}$ uniquely
defined up to sets of measure zero. Then one gets
\[\wassd_\mu^2(\xi,\zeta)=\int_{\mathbb{S}^1} \wassd^2(\xi_x,\zeta_x) \mu(dx)\]
where one integrates the squared Wasserstein metric defined with respect to the
Riemannian metric, that is $|\cdot|$.

There is a natural cone structure on $\mathscr{P}(T\mathbb{S}^1)_\mu$, extending the scalar 
multiplication on the tangent bundle: letting $D_r$ be the 
dilation of ratio $r$ along fibers, acting on $T\mathbb{S}^1$, one defines 
$r\cdot \xi:=(D_r)_\#\xi$.

The exponential map $\exp:T\mathbb{S}^1\to \mathbb{S}^1$ now gives a map
\[\exp_\# : \mathscr{P}(T\mathbb{S}^1)_\mu\to\measures.\]
The point is that not for all 
$\xi\in \mathscr{P}(T\mathbb{S}^1)_\mu$, is there a $\varepsilon>0$ such that 
$t\mapsto \exp_\#(t\cdot\xi)$ defines a geodesic of $\measures$ on 
$[0,\varepsilon)$. Consider for example $\mu=\lambda$, and $\xi$ be defined
by $\xi_x\equiv1$. Then $\exp_\#(t\cdot\xi)=\lambda$ for all $t$: one rotates all
the mass while letting it in place would be more efficient.

The first definition is that $T_\mu$ is the closure in $\mathscr{P}(T\mathbb{S}^1)_\mu$ of the subset
of all $\xi$ such that $\exp_\#(t\cdot\xi)$ defines a geodesic for small 
enough $t$.

%%%%%%%%%%%%%%%%%%%%%%%%%%%%%%%%%%%%%%%%%%%%%%%%%%%%%%%%%%%%%%%%%
\subsection{Another definition of the tangent space}

Let us now give another definition, assuming $\mu$ is atomless.
We denote by $|\cdot|_{L^2(\mu)}$ the norm defined by the measure $\mu$, and
by $|\cdot|_2$ the usual $L^2$ norm defined by the Lebesgue measure
$\lambda$.

Given a smooth 
function $f:\mathbb{S}^1\to\mathbb{R}$, its gradient 
$\nabla f:\mathbb{S}^1\to T\mathbb{S}^1$ can be used to push $\mu$
to an element $\xi_f=(\nabla f)_\#\mu$ of $\mathscr{P}(T\mathbb{S}^1)_\mu$.
 This element has the 
property that $\exp_\#(t\cdot\xi)=(\Id+t\xi_f)_\#\mu$ defines a geodesic for small 
enough $t$, with a time bound depending on 
$\nabla f$ and not on $\mu$. More precisely,
the geodesicness holds as soon as no mass is moved
at a distance exceeding $1/2$, and no element of mass crosses another one,
and these conditions translate to $t (\nabla f)'(x)\geqslant -1$ for all
$x$. This is a particular case of Kantorovich duality, see for example
\cite{Villani2}, especially figure 5.2.

Now, let $L^2_0(\mu)$ be the set of all vector fields $v\in L^2(\mu)$
that are $L^2(\mu)$-approximable by gradient of smooth functions.
Then the image of the map $v\mapsto (\Id,v)_\#\mu$ defined on $L^2_0(\mu)$ 
with value in $\mathscr{P}(T\mathbb{S}^1)_\mu$ is precisely $T_\mu$.
In particular, this means that as soon as $\mu$ is atomless, the disintegration
$(\xi_x)_x$ of an element of $T_\mu$ writes $\xi_x=\delta_{v(x)}$ for some
function $v$ and $\mu$-almost all $x$. Moreover, $v$ is $L^2(\mu)$-approximable
by gradient of smooth functions; note that among smooth vector fields,
gradients are characterized by $\int \nabla f \lambda = 0$.
We shall freely identify the tangent space with $L^2_0(\mu)$ whenever $\mu$
has no atom.

In the important case when $\mu=\rho\lambda$ for some positive
continuous density $\rho$,
a vector field $v\in L^2(\mu)$ is approximable by gradient of smooth functions
if and only if $\int v\lambda = 0$.
We get that in this case, $T_\mu$ can be 
identified with the set of functions $v:\mathbb{S}^1\to \mathbb{R}$
that are square-integrable with respect to $\mu$ and of mean zero 
with respect to $\lambda$.  When $\mu$
is the uniform measure, we write $L^2_0$ instead of $L^2_0(\lambda)$.
Note that if $v\in L^2(\mu)$ has neither its negative part nor its
positive part $\lambda$-integrable, then it can be approximated in
$L^2(\mu)$ norm by gradient of smooth functions, and that
if $\mu$ has not full support, then $L^2_0(\mu)=L^2(\mu)$.

For simplicity, given $v\simeq \xi\in L^2_0(\mu)\simeq T_\mu$ we shall denote
$\exp_\#(t\cdot\xi)$ by $\mu+tv$. In other words,
$\mu+tv=(\Id+tv)_\#\mu$.

This point of view is convenient, in particular because the distance between
exponential curves issued from $\mu$ can be estimated easily:
\[\wassd(\mu+tv,\mu+tw)\underset{t\to 0}\sim t|v-w|_{L^2(\mu)}.\]
 Note that when $v$ is differentiable,
then by geodesicness for $t$ small enough we have
\[\wassd(\mu,\mu+tv) = t |v|_{L^2(\mu)}\]
and not only an equivalence. This will prove useful in the next subsection
where several measures and vector fields will be involved.

%%%%%%%%%%%%%%%%%%%%%%%%%%%%%%%%%%%%%%%%%%%%%%%%%%%%%%%%%%%%%%%%%%%%%%%%%%%%%%%
\subsection{Two properties}

We shall prove that the exponential map can be used to construct
bi-Lipschitz embeddings of small, finite-dimensional balls into $\measures$,
then we shall study how the density of an absolutely continuous
measure evolves when pushed by a small vector field.

The following natural result shall be used in the proof of Theorem 
\ref{theo:almost-invariant}.
\begin{prop}\label{prop:embedding}
Given $\mu\in \measures$ and $(v_1,\ldots,v_n)$
continuous, linearly independent vector fields in $L^2_0(\mu)$,
there is an $\eta>0$ such that the map $B^n(0,\eta)\to\measures$ defined
by $E(a)=\mu+\sum a_i v_i$ is bi-Lipschitz.
\end{prop}
The difficulty is only technical: we already know that $E$ is bi-Lipschitz
along rays and we need some uniformity in the distance estimates to prove
the global bi-Lipschitzness. The continuity hypothesis is not satisfactory
but is all we need in the sequel.

Note that we did not assume that $\mu$ has no atom; when it has, $L^2_0(\mu)$
(still defined as the closure in $L^2(\mu)$ of gradients of smooth functions)
is not the tangent cone $T_\mu\measures$ but only a part of it. Note that
if $v$ is a $C^1$ vector field of vanishing $\lambda$-mean,
$(\mu+tv)_t$ still defines a geodesic as long as $tv'\geqslant -1$.

\begin{proof}
Let $a,b\in B^n$. The plan $(\Id+\sum a_i v_i,\Id+\sum b_i v_i)_\#\lambda$
transports $E(a)$ to $E(b)$
at a cost 
\[\left|\sum (a_i-b_i)v_i\right|_2^2 
\leqslant \left(\sum |v_i|_2^2\right)\, |a-b|^2\]
so that $E$ is Lipschitz.

Up to a linear change of coordinates, we assume that the $v_i$ form
an orthonormal family of $L^2_0(\mu)$. To bound the distance between
$E(a)$ and $E(b)$ from below, we shall design a vector field $\tilde v$
such that pushing $E(a)$ by $\tilde v$ gives a measure close to 
$E(b)$.

Choose $\varepsilon>0$
such that for all $i$ we have 
\[|x-y|\leqslant\varepsilon \Rightarrow |v_i(x)-v_i(y)|\leqslant \frac{1}{4\sqrt{n}}.\]
Assume moreover $\varepsilon<1/8$.

Let $w_i$ be gradient of smooth functions such that
$|v_i-w_i|_\infty\leqslant \varepsilon$.
Let $\eta>0$ be small enough to ensure $2\sqrt{n}\eta\leqslant 1$ and
$w_i'\geqslant -(4n\eta)^{-1}$ fo all $i$.

Fix $a,b\in B^n(0,\eta)$ and introduce two maps defined by
$\psi(y)=y+\sum a_i v_i(y)$ and $\tilde\psi(y)=y+\sum a_i w_i(y)$.
Note that $\tilde\psi'\geqslant 1/2$ so that $\tilde\psi$ is
a diffeomorphism and $\tilde\psi^{-1}$ is $2$-Lipschitz. Let
$\tilde v = \sum (b_i-a_i)v_i\circ\tilde\psi^{-1}$.

On the first hand, given any $y\in\mathbb{S}^1$, we have
\[|\tilde\psi(y)-\psi(y)|\leqslant |a|\left(\sum(w_i(y)-v_i(y))^2\right)^{1/2} 
  \leqslant |a|\sqrt{n}\varepsilon\]
so that
\[|y-\tilde\psi^{-1}\psi(y)|\leqslant 2\sqrt{n}|a|\varepsilon\leqslant \varepsilon\]
and
\[\left|v_i(\tilde\psi^{-1}\psi(y))-v_i(y)\right|\leqslant\frac1{4\sqrt{n}}.\]
It follows that
\[\left|\sum(b_i-a_i)(v_i(\tilde\psi^{-1}\psi(y)) -v_i(y))\right|\leqslant\frac14|b-a|,\]
and therefore
\begin{equation}
\left|\tilde v\circ\psi-\sum(b_i-a_i)v_i\right|_{L^2(\nu)}\leqslant\frac14|b-a|
\label{eq:lip1}
\end{equation}
where $\nu$ could be any probability measure. We shall take
$\nu=\mu+\sum a_i v_i$.

Similarly,
\begin{eqnarray}
|\tilde v|_{L^2(\nu)} &=& \left(\int \tilde v^2(x) \,(\psi_\#\mu)(dx)\right)^{1/2} \nonumber\\
  &=& \left(\int \tilde v^2(\psi x)\,\mu(dx)\right)^{1/2} \nonumber\\
  &=& \left|\sum(b_i-a_i)v_i\tilde\psi^{-1}\psi\right|_{L^2(\mu)} \nonumber\\
  &\geqslant& \frac34\left|\sum(b_i-a_i)v_i\right|_{L^2(\mu)} \nonumber\\
|\tilde v|_{L^2(\nu)}  &\geqslant& \frac34 |b-a|.
\end{eqnarray}

On the other hand, we have
\[ \wassd\left(\mu+\sum a_i v_i, \mu+\sum b_i v_i\right)\geqslant
  \wassd(\nu,\nu+\tilde v)-\wassd\left(\nu+\tilde v,\mu+\sum b_i v_i\right).\]

Let $\tilde w=\sum(b_i-a_i)w_i\circ\tilde\psi^{-1}$. We have
$|\tilde v-\tilde w|_\infty\leqslant \varepsilon |b-a|$.
In particular, $|\tilde w|_{L^2(\nu)}\geqslant\frac58|b-a|$.
The choice of $\eta$ ensures that $\tilde w'\geqslant-1$, so that
\[\wassd(\nu,\nu+\tilde w)=|\tilde w|_{L^2(\nu)}\geqslant \frac58|b-a|.\]
Since $\wassd(\nu+\tilde v,\nu+\tilde w)\leqslant |\tilde v-\tilde w|_\infty$
we get
\begin{equation}
\wassd(\nu,\nu+\tilde v)\geqslant \frac12|b-a|.
\end{equation}
Finally, since $\nu+\tilde v= (\psi+\tilde v \psi)_\#\mu$,
\eqref{eq:lip1} shows that 
\[\wassd\left(\nu+\tilde v,\mu+\sum b_i v_i\right)\leqslant\frac14|b-a|\]
so that
\[ \wassd\left(\mu+\sum a_i v_i, \mu+\sum b_i v_i\right)\geqslant \frac14|b-a|.\]
\end{proof}

\begin{prop}\label{prop:density}
Let $\rho$ be a $C^1$ density and $v:\mathbb{S}^1\to \mathbb{R}$
be a $C^1$ vector field. Then for $t\in\mathbb{R}$ small enough
$\rho\lambda+tv$ is absolutely continuous and its density
$\rho_t$ is continuous and satisfy
\[\rho_t(x) = \rho(x) -t(\rho v)'(x) + o(t)\]
where the remainder term is independent of $x$.
\end{prop}

\begin{proof}
Let $t$ be small enough so that $\Id+tv$ is a diffeomorphism.
By a change of variable, we see that 
\begin{eqnarray*}
\rho_t &=& \frac{\rho}{1+tv'}\circ(\Id+tv)^{-1}\\
       &=& \left(\rho(1-tv')\right)\circ(\Id-tv)+o(t)\\
       &=& \rho-t(\rho'v+v'\rho)+o(t)
\end{eqnarray*}
where the $o(t)$ term depends upon $\rho$
and $v$ but is uniform in $x$.
\end{proof}
Note that the $o(t)$ depends in particular on the
moduli of continuity of $v'$ and $\rho'$ and need not
be an $O(t^2)$ unless $v$ and $\rho$ are $C^2$.

%%%%%%%%%%%%%%%%%%%%%%%%%%%%%%%%%%%%%%%%%%%%%%%%%%%%%%%%%%%%%%%%%%%%%%%%%%%%%%%%
%%%%%%%%%%%%%%%%%%%%%%%%%%%%%%%%%%%%%%%%%%%%%%%%%%%%%%%%%%%%%%%%%%%%%%%%%%%%%%%%
%%%%%%%%%%%%%%%%%%%%%%%%%%%%%%%%%%%%%%%%%%%%%%%%%%%%%%%%%%%%%%%%%%%%%%%%%%%%%%%%
\section{First-order dynamics in the model case}\label{sec:firstorder}

In this section we show that $\Phi_{d\#}$ is (weakly) differentiable at the point 
$\lambda$. Its derivative is an
explicit, simple endomorphism of a Hilbert space, and we shall give a brief
study of its spectrum.

\begin{theo}\label{theo:differential}
Let $\mathscr{L}_d:L^2_0\to L^2_0$ be the linear operator defined by
\[\mathscr{L}_d v(x)= v(x/d)+v((x+1)/d)+\dots+v((x+d-1)/d).\]
 Then $\mathscr{L}_d$ is the
derivative of $\Phi_{d\#}$ at $\lambda$ in the following sense:
for all $v\in L^2_0\simeq T_\lambda$, one has
\[\wassd\left(\Phi_{d\#}(\lambda+tv),\lambda+t\mathscr{L}_d(v)\right)=o(t).\]
\end{theo}
First, we recognize in $\mathscr{L}_d$ a multiple of
the Perron-Frobenius operator of $\Phi_d$,
that is the adjoint of the map $u\mapsto u\circ \Phi$, acting on the space $L^2_0$.
Second, we only get a G\^ateaux derivative, when one would prefer a Fr\'echet one,
that is a formula of the kind
\[\wassd(\Phi_{d\#}(\lambda+v),\lambda+\mathscr{L}_d(v))=o(|v|).\]
However, we shall see that such a uniform bound does not
hold. 
However, one easily gets uniform remainder terms in restriction to any finite-dimensional
subspace of $L^2_0$.

%%%%%%%%%%%%%%%%%%%%%%%%%%%%%%%%%%%%%%%%%%%%%%%%%%%%%%%%%%%%%%%%%%%%%%%%%%%%%%%%%%%
\subsection{Differentiability of $\Phi_{d\#}$}

The main point to prove in the above theorem is the following estimate;
this is where the original article contained a mistake.%
\footnote{More precisely, in \cite{Original} 
the right-hand side of the first inline equation in Lemma 4.2 should be
$\varepsilon t+2^{-3/2}\varepsilon$ rather than 
$(1+2^{-3/2})\varepsilon t$.
This mistake can be corrected by estimating how well 
a piecewise constant density with $k$ pieces of equal length
 can approximate the given density.
Then the issue is moved to the main argument: in order to ultimately get a $o(t)$
remainder, we need to take advantage of the presence of many overlaps (as in Figure \ref{fig:transport}), which only exist
if $k$ increases not too fast with respect to $t$. This can be ensured by adding the regularity hypothesis.
We shall only used Lemma 4.2 for positive $C^1$ densities, so
this hypothesis is harmless.}

\begin{lemm}\label{lemm:composition}
Given a \emph{H\"older continuous} and \emph{positive} density $\rho$, vector fields 
$v_1,\dots,v_n\in L^2(\rho\lambda)$ and positive numbers 
$\alpha_1,\dots \alpha_n$ summing up to $1$, one has
\[\wassd\Big( \rho\lambda+t\smallsum_i \alpha_i v_i, 
  \smallsum_i \alpha_i (\rho\lambda+t v_i)\Big) = o(t).\]
\end{lemm}

The positivity assumption may not be necessary, but at the very least simplifies the proof.

\begin{proof}
We prove the case $n=2$ since the general case can then be deduced by induction.

Let $\varepsilon$ be any positive number, and consider vector fields
$\bar v_i$ ($i=1,2$) that are constant on the intervals 
$[j/k_1,(j+1)/k_1)$ for some $k_1$ and all $j<k_1$ and such that
$\lVert \bar v_i - v_i \rVert_{L^2(\rho\lambda)} \le \varepsilon$.
Note that $k_1$ and the $\bar v_i$ are chosen to depend only on
$\varepsilon$, not on $t$; in particular 
$\lVert \bar v_1-\bar v_2\rVert_{\infty}$ is finite and independent
of $t$.

Now consider any value of $t$, to be taken small enough a few times below.
Let $k=k(t)$ be a multiple of $k_1$ having the
magnitude of $(1/t)^{1/(1+\beta/2)}$ where $\beta$ is the H\"older exponent of
$\rho$, say $kt^{1/(1+\beta/2)}\in [1,2]$.

We define $\bar\rho$ as the density that is constant on each $I_j=[\frac jk,\frac{j+1}k)$,
of value $\bar\rho_j:=k\int_{I_j} \rho \,\dd\lambda$. Denoting by $C$
the H\"older constant of $\rho$, we get
\[\Vert\rho-\bar\rho\rVert_\infty \le Ck^{-\beta}\]
We denote by $\bar v_i(j)$ the value of $\bar v_i$ on $I_j$; observe
that when $t$ is small, these values are the same on many successive intervals since
$k$ is much larger than $k_1$.

Let us first bound above $\wassd(\rho\lambda,\bar\rho\lambda)$.
We consider the monotone rearrangement fixing $0$ as transport plan;
by definition of $\bar\rho$, it preserves each $I_i$.
To simplify notation, let us bound the cost due to the mass located in $I_0$, 
the other intervals behaving in exactly the same way.
The cumulative
distribution functions of $\rho\lambda$ and $\bar\rho\lambda$
are given by
\[F(x)=\int_0^x \rho \,\dd\lambda \quad\mbox{and}\quad 
  G(x)=x\bar\rho_0.\]
The monotone rearrangement is given on $I_0$ by the map $T = G^{-1}\circ F$,
so that the contribution of $I_0$ to its cost is
\begin{align*}
\int_0^{\frac1k} |T(x)-x|^2 \rho(x)\,\dd x 
  &= \int_0^{\frac1k} \bigg|\frac1{\bar\rho_0}\int_0^x \rho \,\dd\lambda
    -\frac1{\bar\rho_0}\int_0^x \bar\rho_0 \,\dd\lambda \bigg|^2  \rho(x)\,\dd x \\
  &\le \int_0^{\frac1k} \Big|\frac1{\bar\rho_0}\int_0^x 
        |\rho-\bar\rho_0| \,\dd\lambda\Big|^2  \rho(x)\,\dd x \\
  &\le \int_0^{\frac1k} \frac{C^2 x^2}{\bar\rho_0^2 k^{2\beta}}
       \rho(x)\,\dd x \\
  &\le \frac{C^2}{\bar\rho_0 k^{3+2\beta}}
\end{align*}
Since the mass lying in $I_0$ is $\bar\rho_0/k$ (for both densities),
the ratio cost per mass is bounded above by 
\[\frac{C^2}{\bar\rho_0^2 k^{2+2\beta}} \le\frac{C^2}{(\min\rho)^2 \, k^{2+2\beta}}.\]
Since this holds in all intervals $I_i$, the overall cost is bounded by
the same value, so that 
\[\wassd(\rho\lambda,\bar\rho\lambda)\le \frac{C}{\min\rho} \frac1{k^{1+\beta}}.\]

The same argument also yields
\[\wassd(\rho\lambda + v,\bar\rho\lambda + v) \le\frac{C}{\min\rho} \frac1{k^{1+\beta}} \]
for any vector field $v$ which is constant on each $I_j$: indeed, if $\Pi$ is a transport plan
from $\rho\lambda$ to $\bar\rho\lambda$, then
$(\Id+v,\Id+v)_\#\Pi$ is a transport plan from
$\rho\lambda + v$ to $\bar\rho\lambda + v$ whose cost is not greater than the cost of $\Pi$
(for each bit of mass moved from $x$ to $T(x)$ by $\Pi$, this new plan moves the 
same amount of mass from $x+v(x)$ to $T(x)+v(T(x))$; the hypothesis that $v$ is
constant on each $I_j$ then ensures that $v(T(x))=v(x)$).
Applying this to $v=t\smallsum\alpha_i \bar v_i$ we get
\[\wassd(\rho\lambda+t\smallsum\alpha_i \bar v_i,\bar\rho\lambda+t\smallsum\alpha_i \bar v_i)
  \le \frac{C}{\min\rho} \frac1{k^{1+\beta}}.\]
Applying the same reasonning to each $v=t \bar v_i$ separately and concatenating
the corresponding transport plan also yields
\[\wassd(\smallsum \alpha_i(\rho\lambda+t \bar v_i),\smallsum\alpha_i (\bar\rho\lambda+t \bar v_i))
  \le \frac{C}{\min\rho} \frac1{k^{1+\beta}}.\]

We will now prove the bound
\[\wassd(\bar\rho\lambda+t\smallsum\alpha_i \bar v_i,
  \smallsum \alpha_i(\bar\rho\lambda+t\bar v_i)) \le 
  t^{3/2}k^{1/2}\lVert\bar v_1-\bar v_2\rVert_\infty^{3/2}.\]

For this, on each $I_j$ the construction pictured in
Figure \ref{fig:transport} gives a transport plan from
$(\Id+t(\alpha_1v_1(j)+\alpha_2v_2(j)))_\# \bar\rho_j\lambda_{|I_j}$ to 
$\alpha_1(\Id+tv_1(j))_\#\bar\rho_j\lambda_{|I_j}
+\alpha_2(\Id+tv_2(j))_\#\bar\rho_j\lambda_{|I_j}$
whose cost is at most
$t^3\bar\rho_j|\bar v_1(j)-\bar v_2(j)|^3$ for a contribution to the
mass of $\bar\rho_j/k$. More precisely, temporarily denoting by
$\rho$, $v_1$ and $v_2$
the values taken by the functions $\bar\rho$ and $\bar v_i$ on $I_j$,
one simply let the common mass in 
place and moves at each side a mass $\alpha_1\alpha_2\rho|v_1-v_2|t$
by a distance at most $|v_1-v_2|t$;
this is not optimal but sufficient for our purpose.

\begin{figure}[tbp]\begin{center}
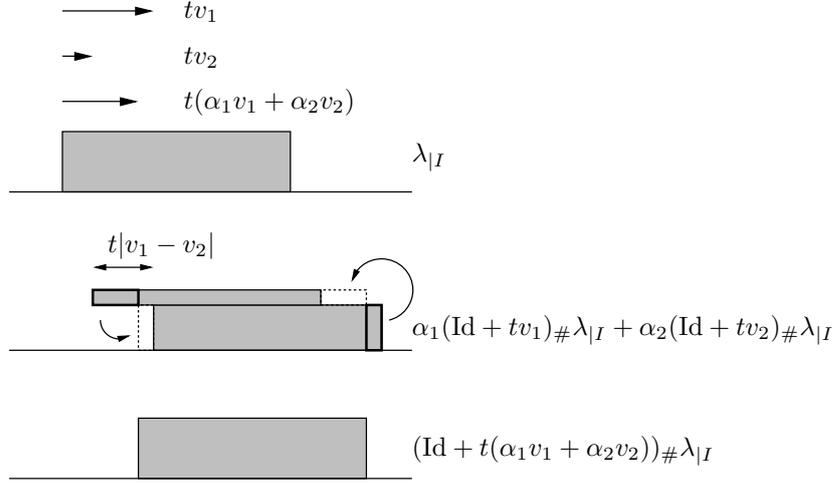
\caption{The cost of this transport plan has the order of magnitude $t^3$}
\label{fig:transport}
\end{center}\end{figure}

The last estimates we need are
\[\wassd\big(\rho\lambda+t\smallsum\alpha_i v_i
,\rho\lambda+t\smallsum\alpha_i \bar v_i\big)\le \varepsilon t\]
and
\[\wassd(\smallsum \alpha_i(\rho\lambda+t \bar v_i),
  \smallsum \alpha_i(\rho\lambda+t v_i))\le \varepsilon t.\]
They are both obtained, as in the original proof, by observing that
for any measure $\mu$ and any vector fields $v,\bar v$ in $L^2(\mu)$,
the transport plan $(\Id+v,\Id+\bar v)_\#\mu$ has cost exactly
$\lVert v-\bar v\rVert_{L^2(\mu)}$.

Using the triangle inequality to combine all these estimates, we get
\begin{multline*}
\wassd\Big(\rho\lambda+t\smallsum \alpha_i v_i, \smallsum \alpha_i(\rho\lambda+t v_i)\Big)\\
  \le 2\varepsilon t+ t^{\frac32}k^{\frac12}\lVert\bar v_1-\bar v_2\rVert_\infty^{\frac32}
       +\frac{2C}{\min\rho} \frac1{k^{1+\beta}}\\
  = \mathrlap{2\varepsilon t 
  + O\big(t^{\frac{4+3\beta}{4+\beta}}\big)+O\big(t^{\frac{2+2\beta}{2+\beta}}\big)=o(t).}
   \phantom{2\varepsilon t+ t^{\frac32}k^{\frac12}\lVert\bar v_1-\bar v_2\rVert_\infty^{\frac32}
       +\frac{2C}{\min\rho} \frac1{k^{1+\beta}}}
\end{multline*}
\end{proof}

\begin{proof}[Proof of Theorem \ref{theo:differential}]
Recall that 
\begin{align*}
\Phi_{d\#}(\lambda+tv)=&\frac1d\big(\lambda+dt\,v(\cdot/d)\big)
  +\frac1d\big(\lambda+dt\,v((\cdot+1)/d)\big)\\
& +\dots+ \frac1d\big(\lambda+dt\,v((\cdot+d-1)/d)\big)
\end{align*}
and apply the preceding lemma.
\end{proof}

Let us prove that we cannot hope for the Fr\'echet differentiability of 
$\Phi_{d\#}$. We only treat the case $d=2$ for simplicity.

\begin{prop}
For all positive $\varepsilon$, there is a vector field
$v\in L^2_0$ that satisfies the following:
\begin{enumerate}
\item\label{enumi:a} $|v|_2\leqslant \varepsilon$,
\item\label{enumi:b}  $\mathscr{L}_2 v=0$ so that $\lambda+\mathscr{L}_2 v=\lambda$, and
\item\label{enumi:c} $\wassd\left(\Phi_{2\#}(\lambda+v),\lambda\right)\geqslant c\varepsilon$
\end{enumerate}
for some constant $c$ independent of $\varepsilon$ and $v$.
\end{prop}

\begin{proof}
Let $k$ be a positive integer, to be precised later on. Let $v$ be the
piecewise affine map defined as follows (see figure \ref{fig:example}): 
$v(x)=1/(4k)-y$ when  $x=i/(2k) + y$ with $y\in[0,1/(2k))$
and $0\leqslant i<k$ an integer, and $v(x)=-1/(4k)+y$ when 
$x=i/(2k) + y$ with $y\in[0,1/(2k))$ and $k\leqslant i<2k$.
We have $|v|_2^2=(4k)^{-2}/3$ so that taking 
$k\geqslant \frac{\sqrt{3}}4\varepsilon^{-1}$ ensures point \ref{enumi:a}.
Moreover, \ref{enumi:b} is straightforward, and we are left to prove that
if $k$ is of order $\varepsilon^{-1}$, then property \ref{enumi:c} holds.

On any small enough interval $I$, if $w$ is an affine function of slope $-1$
with a zero at the center of $I$, then $\lambda_{|I}+w$ is a Dirac mass at
the center of $I$ (each element of mass is moved to the center). If $w$
has slope $1$, then the mass moves in the other direction, and $\lambda_{|I}+w$
is uniform of density $1/2$ on the interval $I'$ having the same center than
$I$ and twice as long. By combining these two observations, one deduces that
\[\mu:=\Phi_{2\#}(\lambda+v)=1/2\lambda
  +\sum_{i=1}^k \frac1{2k}\delta_{\frac{i-1/2}{k}}.\] 

\begin{figure}[htp]\begin{center}
\includegraphics{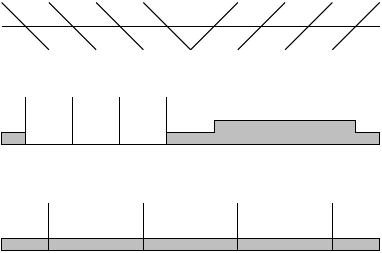}
\caption{The case $k=4$. Up: the graph of $v$; middle: $\lambda+v$; down:
$\Phi_\#(\lambda+v)$.}\label{fig:example}
\end{center}\end{figure}

Each interval of the form $I_i=[(i-5/8)/k,(i-3/8)/k)$ is given by $\lambda$
a mass $1/(4k)$. The discrete part of $\mu$ consists in a Dirac mass of weight 
$1/(2k)$ at the center of each $I_i$. Any transport plan from $\mu$ to $\lambda$
must therefore move a mass at least
$1/(4k)$ from each of these Dirac masses to the outside of $I_i$, so that
a total mass at least $1/4$ has to move a distance at least $1/(8k)$. From this
it follows that $\wassd(\lambda,\mu)\geqslant 1/(16k)$. When $k$ is chosen with the
order of $\varepsilon^{-1}$, this distance has at least the order of 
$\varepsilon$, as required. 
\end{proof}

%%%%%%%%%%%%%%%%%%%%%%%%%%%%%%%%%%%%%%%%%%%%%%%%%%%%%%%%%%%%%%%%%%%%%%%%%%%%%%%
\subsection{Spectral study of $\mathscr{L}_d$}

Let us compute the spectrum of 
$\mathscr{L}_d=D_\lambda(\Phi_{d\#})$.  The following proposition is very 
elementary and 
not new, but we produce a proof for the sake of completeness.

\begin{prop}\label{prop:spec}
A number $\alpha$ is an eigenvalue of $\mathscr{L}_d$ if and only if
$|\alpha|<d$. Moreover, each eigenvalue has an infinite-dimensional eigen\-space.
Last, the spectrum of $\mathscr{L}_d$ is the closed disc of radius $d$.
\end{prop}

The proof of Proposition \ref{prop:spec} consist simply in using Fourier series
to show that (up to a multiplicative constant) $\mathscr{L}_d$ is conjugated to a 
countable product of the shift on $\ell^2(\mathbb{N})$.

\begin{proof}
Let $c_k$ denote the function $x\mapsto \cos(2\pi k x)$ defined on the
circle, and $s_k:x\mapsto \sin(2\pi k x)$. Then it is readily checked that
$\mathscr{L}_d\, c_k=\mathscr{L}_d\,s_k=0$ when $d$ does not divide $k$, and 
$\mathscr{L}_d\, c_k=dc_{k/d}$, $\mathscr{L}_d\,s_k=d s_{k/d}$ when $d|k$.

Let $\sigma$ be the shift of the Hilbert space $\ell^2=\ell^2(\mathbb{N})$ of 
$\mathbb{N}$-indexed
square integrable sequences: if $\underline{x}=(x_0,x_1,x_2,\ldots)$ then
$\sigma\underline{x}=(x_1,x_2,x_3,\ldots)$. Let $\sigma^{\mathbb{N}}$
be the direct product of $\sigma$, acting diagonaly on the space $(\ell^2)^{(\mathbb{N})}$
of sequences $X=(\underline{x}^0,\underline{x}^1,\underline{x}^2,\ldots)$ such that
$\underline{x}^i\in\ell^2$ and $\sum |\underline{x}^i|_2^2<\infty$.
Then the map $\Psi : (\ell^2)^{(\mathbb{N})} \to L^2_0$ defined by
\begin{align*}
\Psi(X)= &\sum_{i,j\in\mathbb{N}} x_j^{2(d-1)i} c_{(di+1)d^j} 
  + x_j^{2(d-1)i+1} c_{(di+2)d^j}\\
  & \qquad+\dots+x_j^{2(d-1)i+d-2} c_{(di+d-1)d^j} \\
  & + x_j^{2(d-1)i+d-1} s_{(di+1)d^j}
  +x_j^{2(d-1)i+d} s_{(di+2)d^j} \\
  & \qquad+\dots+x_j^{2(d-1)i+2d-3} s_{(di+d-1)d^j}
\end{align*}
is an isomorphism (and even an isometry) that intertwins $\sigma^{\mathbb{N}}$
and $\frac1d\mathscr{L}_d$. The spectral study of $\mathscr{L}_d$ therefore reduces
to that of $\sigma$.

A non-zero eigenvector of $\sigma$, associated to an eigenvalue $\alpha$,
must have the form $(x,\alpha x,\alpha^2 x,\ldots)$ with $x\neq 0$. Such a
sequence is square integrable if and only if $|\alpha|<1$. Moreover
the operator norm of $\sigma$ is $1$, so that its complex spectrum is a subset
of the closed unit disc. Since the spectrum is closed, and contains
the set of eigenvalues, it is equal to the closed unit disc.
\end{proof}

%%%%%%%%%%%%%%%%%%%%%%%%%%%%%%%%%%%%%%%%%%%%%%%%%%%%%%%%%%%%%%%%%%%%%%%%%%%%%%%%%%%%%%%%%%%
\subsection{Discussion of the non-Fr\'echet differentiability}

The coun\-ter-example to the Fr\'echet differentiability of $\Phi_\#$
at $\lambda$ has high total variation, and it is likely that using a norm
that controls variations (e.g. a Sobolev norm) on (a subspace of) $T_\lambda$
shall provide a uniform error bound. 

Moreover, up to multiplication by $d$ the derivative $\mathscr{L}_d$ is 
the Perron-Frobenius operator of $\Phi_d$,
and such operators have
far more subtle spectral properties when defined over Sobolev spaces.

For these two reasons, it seems that one could search for a modification of
optimal transport that would give a manifold structure to $\measures$,
 in such a way that $T_\lambda$ identifies with a Sobolev space. A way to 
achieve this could be to penalize not only the distance by which a transport
plan moves mass, but also the distorsion, that is the variation of the pairwise
distances of the elements of mass. This should impose more regularity to optimal
transport plans.

%%%%%%%%%%%%%%%%%%%%%%%%%%%%%%%%%%%%%%%%%%%%%%%%%%%%%%%%%%%%%%%%%%%%%%%%%%%%%%%%
%%%%%%%%%%%%%%%%%%%%%%%%%%%%%%%%%%%%%%%%%%%%%%%%%%%%%%%%%%%%%%%%%%%%%%%%%%%%%%%%
%%%%%%%%%%%%%%%%%%%%%%%%%%%%%%%%%%%%%%%%%%%%%%%%%%%%%%%%%%%%%%%%%%%%%%%%%%%%%%%%
\section{First-order dynamics for general expanding maps}\label{sec:general}

In this section, we consider a general map $\Phi:\mathbb{S}^1\to\mathbb{S}^1$,
assumed to be $C^2$ and expanding, {\it i.e.} $|\Phi'|>1$. Such a map is a self-covering,
and has a unique absolutely continuous invariant measure 
(see e.g. \cite{Katok-Hasselblatt})
which has a positive and $C^1$ density \cite{Krzyzewski}, denoted by $\rho$. The measure itself is denoted
by $\rho\lambda$.
Note that as sets, $L^2(\rho\lambda)=L^2$, although they differ as Hilbert spaces.
All integrals where the variable is implicit are with respect to the Lebesgue measure $\lambda$.

The result is as follows.
\begin{theo}\label{theo:expanding}
The map $\Phi_\#$ has a G\^ateaux derivative
$\mathscr{L} : L^2_0(\rho\lambda) \to  L^2_0(\rho\lambda)$ at $\rho\lambda$,
given by
\[\mathscr{L}v(x) = \sum_{y\in\Phi^{-1}(x)} \frac{\rho(y)}{\rho(x)} v(y)
              -\frac{\int{v\Phi'\frac{\rho}{\rho\circ\Phi}}}{\rho(x)\int1/\rho}
\]
Moreover the adjoint operator of $\mathscr{L}$ in $L^2_0(\rho\lambda)$
is given by
\[\mathscr{L}^* u = \Phi'\, u\circ\Phi.\]
\end{theo}

%%%%%%%%%%%%%%%%%%%%%%%%%%%%%%%%%%%%%%%%%%%%%%%%%%%%%%%%%%%%%%%%%%%%%%%%%%%%%%%%
\subsection{Proof of Theorem \ref{theo:expanding}}

First, as in the case of $\Phi_{d\#}$, Lemma \ref{lemm:composition} shows that
for $v\in L^2_0(\rho\lambda)$,
\begin{equation}
d\left(\Phi_\#(\rho\lambda+tv), \rho\lambda+t\tilde{\mathscr{L}}v\right)=o(t)
\label{eq:tildoperator}
\end{equation}
where 
\[\tilde{\mathscr{L}}v(x) = \sum_{y\in\Phi^{-1}(x)} \frac{\rho(y)}{\rho(x)} v(y)\]
is the first term in the expression of $\mathscr{L}$.
In words, each of the antecedents of $x$ gives a contribution to the local displacement
of mass that is proportional to $v(y)$ and to $\rho(y)$.

This seems very similar to the case of $\Phi_{d\#}$, except that $\tilde{\mathscr{L}}$ need
not to map $L^2_0(\rho\lambda)$ to itself! Let us stress, once again, that the condition
that $v\in L^2_0(\rho\lambda)$ has mean zero is to be understood \emph{with respect to
the uniform measure} $\lambda$, since it translates the \emph{metric} property of being (close to)
the gradient of a smooth function. This does not prevent Equation \eqref{eq:tildoperator}
to make sense, but shows that $\tilde{\mathscr{L}}v$ cannot be considered as the
directional derivative of $\Phi_\#$ since it does not belong to $T_{\rho\lambda}=L^2_0(\rho\lambda)$.
In fact, we shall see that there is another vector field, that lies in $L^2_0(\rho\lambda)$ and
gives the same pushed measure (at least at order $1$).

\begin{prop}\label{prop:centering}
Given $\tilde w\in L^2(\rho\lambda)$ and assuming that $\tilde w$ is $C^1$, there 
is a $C^1$
vector field $w\in L^2_0(\rho\lambda)$ such that
$\wassd(\rho\lambda+t\tilde w,\rho\lambda+tw)=o(t)$. Moreover, $w$ is given by
\[w=\tilde w+\frac{\int \tilde w}{\rho\int 1/\rho}.\]
\end{prop}

\begin{proof}
This is a direct application of Proposition \ref{prop:density}: 
we search for a $w$ such that $(\rho w)'=(\rho\tilde w)'$,
so that the densities $\rho_t$ and $\tilde\rho_t$ of $\rho\lambda+tw$
and $\rho\lambda+t\tilde w$ are $L^\infty$ and therefore $L^1$ close
one to the other. This ensures that 
$\wassd(\rho\lambda+t\tilde w,\rho\lambda+tw)\leqslant |\rho_t-\tilde\rho_t|=o(t)$.

But there exists exactly one vector field $w$ that is $C^1$, 
has mean zero, and such that
$(\rho w)'=(\rho\tilde w)'$: it is given by the claimed formula.
\end{proof}

Note that we did not bother to prove the unicity of $w$: Gigli's construction
shows that the first order
perturbation of the measure (with respect to the $L^2$ Wasserstein metric)
characterizes a tangent vector 
in $T_\mu$, see Theorem 5.5 in \cite{Gigli}.

Now if one considers the ``centering'' operator 
$\mathscr{C}:L^2(\rho\lambda)\to L^2_0(\rho\lambda)$
defined by 
\[\mathscr{C} v= v-\frac{\int v}{\rho\int 1/\rho},\]
the derivative of $\Phi_\#$ at $\rho\lambda$ is given by the composition 
$\mathscr{C}\tilde{\mathscr{L}}$.
Indeed, the previous proposition shows this for a $C^1$ argument, but 
$C^1$ vector fields are dense
in $L^2_0(\rho\lambda)$ and the involved operators are continuous
in the $L^2(\rho\lambda)$ topology.

To get the expression of $\mathscr{L}$ given in Theorem \ref{theo:expanding}, one 
only needs a change of variable:
denoting by $\Phi_i^{-1}$ ($i=1,2,\ldots, d$) the right inverses to $\Phi$ 
that are onto intervals $[a_1=0,a_2), [a_2,a_3), \ldots,
[a_d,a_{d+1}=1)$ one has
\begin{eqnarray*}
\int\tilde{\mathscr{L}}v 
  &=& \sum_i \int \frac{\rho\circ\Phi_i^{-1}}{\rho} v\circ\Phi_i^{-1} \\
  &=& \sum_i \int_{a_i}^{a_{i+1}} \frac{\rho}{\rho\circ\Phi} v \Phi' \\
  &=& \int  v \Phi'\frac{\rho}{\rho\circ\Phi}.
\end{eqnarray*}

The computation of the adjoint is a similar change of variable that we omit. 
Note that the adjoint
of the extension to $L^2(\rho\lambda)$ of $\mathscr{L}$ (with the same expression) is
\[u\mapsto \Phi'\, u\circ\Phi - \frac{\Phi'\int u}{\rho\circ\Phi \int{1/\rho}}\]
and the second term vanishes when $u$ is in $L^2_0(\rho\lambda)$. 
The first term is also the adjoint in $L^2(\rho\lambda)$
of $\tilde{\mathscr{L}}$, and this adjoint preserves $L^2_0(\rho\lambda)$. 
In other words,
$\mathscr{L}$ is the adjoint in $L^2_0(\rho\lambda)$ of the adjoint in
$L^2(\rho\lambda)$ of $\tilde{\mathscr{L}}$.
An interesting feature of the expression of $\mathscr{L}^*$ is that it does not 
involve the invariant measure.

%%%%%%%%%%%%%%%%%%%%%%%%%%%%%%%%%%%%%%%%%%%%%%%%%%%%%%%%%%%%%%%%%%%%%%%%%%%%%%%%
\subsection{Spectral study}

Even if $\mathscr{L}$ is not a multiple of the Perron-Frobenius operator
of $\Phi$, its first term $\tilde{\mathscr{L}}$
is a weighted transfert operator, with weight
$g=\frac{\rho}{\rho\circ\Phi}$. According to Theorem 2.5 in \cite{Baladi},
every number of modulus less than $R_g=\lim_n(\sup\tilde{\mathscr{L}}^n1)^{1/n}$
is an eigenvalue of infinite multiplicity with continuous eigenfunctions.

\begin{prop}
We have $R_g\geqslant \min \Phi'>1$, and therefore
there is an infinite linearly independent family $(v_i)_i$
of continuous functions in $L^2_0(\rho\lambda)$ such
that $\mathscr{L}v_i=v_i$.
\end{prop}

\begin{proof}
Let $m=\min \Phi'$: we have $m>1$ and, since $\rho\lambda$ is
invariant,
\[\rho(x)=\sum_{y\in\Phi^{-1}(x)}\frac{\rho(y)}{\Phi'(y)}
  \leqslant \frac1m\sum_{y\in\Phi^{-1}(x)} \rho(y) \]
It follows that for all positive continuous function $f$,
\[\tilde{\mathscr{L}}f=\sum_{y\in\Phi^{-1}(x)}\frac{\rho(y)}{\rho(x)}f(y)
  \geqslant m|\inf f|;\]
in particular, $R_g\geqslant m>1$ and there is a linearly independent
infinite family $u_0,u_1,\ldots,u_i\ldots$
of continuous $1$-eigenfunctions of $\tilde{\mathscr{L}}$.
If not all $u_i$ have mean $0$ (with respect to Lebesgue's measure
$\lambda$), assume the mean of $u_0$ is not zero and
let $v_i=u_i-\alpha_i u_0$ where $\alpha_i$ is chosen such that
$\int v_i\lambda=0$. Otherwise, simply put $v_i=u_i$.

Now, since $\tilde{\mathscr{L}}v_i=v_i$ and $v_i$ has mean zero,
we get $\mathscr{L} v_i=\tilde{\mathscr{L}}v_i=v_i$.
\end{proof}

In the same way, we see that all numbers less than $m>1$ are eigenvalues
of $\mathscr{L}$ (with infinite multiplicity and continuous eigenfunctions).

%%%%%%%%%%%%%%%%%%%%%%%%%%%%%%%%%%%%%%%%%%%%%%%%%%%%%%%%%%%%%%%%%%%%%%%%%%%%%%%%
%%%%%%%%%%%%%%%%%%%%%%%%%%%%%%%%%%%%%%%%%%%%%%%%%%%%%%%%%%%%%%%%%%%%%%%%%%%%%%%%
%%%%%%%%%%%%%%%%%%%%%%%%%%%%%%%%%%%%%%%%%%%%%%%%%%%%%%%%%%%%%%%%%%%%%%%%%%%%%%%%
\section{Nearly invariant measures}\label{sec:nearly-invariant}

In this section we prove Theorem \ref{theo:almost-invariant} and Proposition
\ref{prop:atomless}.

%%%%%%%%%%%%%%%%%%%%%%%%%%%%%%%%%%%%%%%%%%%%%%%%%%%%%%%%%%%%%%%%%%%%%%%%%%%%%%%%
\subsection{Construction}

Fix some positive integer $n$ and let $v_1,\ldots,v_n$ be continuous,
linearly independent eigenfunctions for 
$\mathscr{L}=D_{\rho\lambda}(\Phi_\#)$.

For all $a=(a_1,\ldots,a_n)\in B^n(0,\eta)$, define
$E(a)=\rho\lambda+\sum_i a_i v_i\in\measures$ and using Proposition
\ref{prop:embedding}, choose $\eta$ small
enough to ensure that $E$ is bi-Lipschitz. Then define
$F(a)=E(\eta a)$ on the unit ball $B^n$.

\begin{prop}
We have
\[\wassd\big(\Phi_\#(F(a)),F(a)\big) = o(|a|)\]
and, as a consequence, for all $\varepsilon>0$
and all integer $K$, there is a radius $r$ 
such that for all $k\leqslant K$
and all $a\in B^n(0,c)$ the following holds:
\[\wassd\big(\Phi_{d\#}^k(F(a)),F(a)\big)\leqslant \varepsilon |a|.\]
\end{prop}

\begin{proof}
Since we have restricted ourselves to a finite-dimensional space, 
we have $\wassd\big(\Phi_\#(\rho\lambda+\eta\sum a_iv_i),
  \rho\lambda+\eta\sum a_i\mathscr{L}(v_i)\big) = o(|a|)$
and, since $\mathscr{L}(v_i)=v_i$, we get
$\wassd\big(\Phi_\#(F(a)),F(a)\big) = o(|a|)$.

The second inequality follows easily. The map $\Phi_\#$ is $L$-Lipschitz for some $L>1$
($L=d$ in the model case, $L>d$ otherwise). For all $\varepsilon>0$
and for all integer $K$, let $r>0$ be small enough to ensure that
\[|a|<\delta\Rightarrow \wassd\big(\Phi_\#(F(a)),F(a)\big) 
  \leqslant \frac{L-1}{L^{k-1}-1}\varepsilon |a|.\]
Then
\begin{eqnarray*}
\wassd\big(\Phi_{\#}^k(F(a)),F(a)\big) 
  &\leqslant& \sum_{\ell=1}^{k-1} \wassd\big(\Phi_{\#}^\ell(F(a)),
       \Phi_{\#}^{\ell-1}(F(a))\big)\\
  &\leqslant& \sum_{\ell=1}^{k-1} L^{\ell-1} \wassd\big(\Phi_{d\#}(F(a)),F(a)\big)\\
  &\leqslant& \varepsilon |a|.
\end{eqnarray*}
\end{proof}

This ends the proof of Theorem \ref{theo:almost-invariant}. It would be 
interesting to have explicit control on $r$ in terms of $\varepsilon$,
$n$ and $K$, and in particular to replace the $o(|a|)$
by a $O(|a|^\alpha)$ for some $\alpha>1$. This seems uneasy because,
even in the model
case where $v_i$ are explicit, we can approximate them by $C^\infty$
vector fields $w_i$ with a good control on $(-w_i')^{-1}$ and
$w'$, but only bad bounds on $w''$ (and therefore on the modulus of continuity
of $w'$).

%%%%%%%%%%%%%%%%%%%%%%%%%%%%%%%%%%%%%%%%%%%%%%%%%%%%%%%%%%%%%%%%%%%%%%%%%%%%%%%%
\subsection{Regularity}

Let us prove that given $\mu$ an atomless measure and
$v\in L^2_0(\mu)$ (or, indifferently, $v\in L^2(\mu)$), for all but
countably many values of the parameter $t$, the measure $\mu+tv$
has no atom.

\begin{proof}[Proof of Proposition \ref{prop:atomless}]
By a line in $T\mathbb{S}^1\simeq \mathbb{S}^1\times \mathbb{R}$,
we mean the image of a non-horizontal line of $\mathbb{R}^2$ by
the quotient map $(x,y)\mapsto (x \mod 1,y)$. We sometimes
refer to a line by an equation of one of its lifts in $\mathbb{R}^2$.

The measure $\mu+tv$ has an atom at $s$ if and only if
the measure $\Gamma=(\Id,v)_\#\mu$ defined on $T\mathbb{S}^1$
gives a positive mass to the line $(x+ty=s)$. Since
$\mu$ has no atom, neither does $\Gamma$, and since two lines intersect in a
countable set, the intersection of two lines is  $\Gamma$-negligible.
It follows that there can be at most $n$ different lines that are given
a mass at least $1/n$ by $\Gamma$. In particular, at most countably many lines
are given a positive mass by $\Gamma$, and the result follows.
\end{proof}

For a general $L^2$ vector field, we cannot hope for more.
The following folklore example shows a $L^2_0$ function such that
$\lambda +tv$ is stranger to $\lambda$ for almost all $t$.
\begin{exem}
Let $K$ be a four-corner Cantor set of $\mathbb{R}^2$.
More precisely, $A,B,C,D$ are the vertices of a square,
$S_A,S_B,S_C,S_D$ are the homotheties of coefficient $1/4$ centered
at these points, and $K$ is the unique fixed point of the map
defined on compact sets $M\subset\mathbb{R}^2$ by
\[\mathscr{S}(M)=S_A(M)\cup S_B(M)\cup S_C(M)\cup S_D(M).\]
The Cantor set $K$ projects on a well-chosen line to an interval,
see figure \ref{fig:Cantor}, while in almost all directions
it projects to $\lambda$-negligible sets, see e.g. 
\cite{Peres-Simon-Solomyak} for a proof.
Choose the square so that $K$ projects vertically to $[0,1]$ (identified
to $\mathbb{S}^1$), and for $x\in[0,1]$ define $v(x)$ as the least $y$
such that $(x,y)\in K$. Then $v$ is $L^2$ and, up to a vertical translation,
we can even assume that $v\in L^2_0$. But for almost all $t$,
the measure $\lambda+tv$ is concentrated into a negligible set.
\end{exem}

\begin{figure}\begin{center}
\includegraphics[scale=.5]{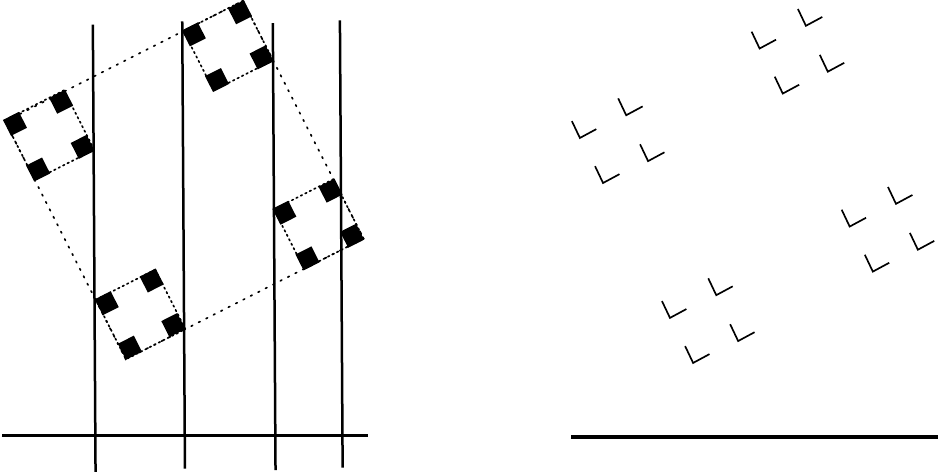}
\caption{A square Cantor set that projects vertically to a segment,
  but projects in almost all directions to negligible sets.
  On the right, an approximation of the graph of the function $v$.}
  \label{fig:Cantor}
\end{center}\end{figure}

%%%%%%%%%%%%%%%%%%%%%%%%%%%%%%%%%%%%%%%%%%%%%%%%%%%%%%%%%%%%%%%%%%%%%%%%%%%%%%%%
%%%%%%%%%%%%%%%%%%%%%%%%%%%%%%%%%%%%%%%%%%%%%%%%%%%%%%%%%%%%%%%%%%%%%%%%%%%%%%%%
\section{The infinitesimal version of Furstenberg's theorem}

In this section we prove Theorem \ref{theo:main}
and Corollary \ref{coro:main}.

\begin{proof}[Proof of Theorem \ref{theo:main}]
We only need to look closely at the expressions of 
$\mathcal{L}_d:=D_\lambda \Phi_{d\#}$.
As indicated in the proof of Proposition \ref{prop:spec},
setting $c_k(x)=\cos(2\pi k x)$ and $s_k(x)=\sin(2\pi k x)$, we
have
\[\mathcal{L}_d(c_k)=\mathcal{L}_d(s_k)=0\]
when $k$ is not a multiple of $d$, and 
\[\mathcal{L}_d(c_{md})=d c_m,\quad \mathcal{L}_d(s_{md})=d s_m\]
when
$m$ is a positive integer. Note that the union of $(c_k)_{k\ge 1}$ and $(s_k)_{k\ge 1}$
is a Hilbert basis of $T_\lambda \mathcal{P}(\mathbb{S}^1)$.

It follows that the $1$-eigenspace $E_d$ of $\mathcal{L}_d$ is generated by the functions
\[\sum_{m\ge0} d^{-m} c_{d^m j} \quad \mbox{and} \quad \sum_{m\ge0} d^{-m} s_{d^m j}\] 
where $j$ runs over positive integers not multiple of $d$.

Taking intersections, we first see that $E_2\cap E_3$ is generated by the functions
\[\sum_{n,m\ge 1} 2^{-n}3^{-m} c_{2^n3^m j} \quad \mbox{and} \quad
  \sum_{n,m\ge 1} 2^{-n}3^{-m} s_{2^n3^m j} \]
where $j$ is any positive integer prime with $6$. In particular, 
$E_2\cap E_3$ is infinite-dimensional.

We also get that $\cap_{d\ge 2} E_d$ is generated by
\[ \sum_{p\ge 1} p^{-1} c_p \quad \mbox{and} \quad  \sum_{p\ge 1} p^{-1} s_p \]
and is thus $2$-dimensional. Note that these functions are indeed 
in $L^2(\lambda)$,
and therefore represent tangent vectors at $\lambda$.
\end{proof}

The proof of the corollary is also easy; it relies mostly on the following
pointwise version of the continuity equation. We do not claim any novelty in
the Lemma below, but still provide a simple proof of the simple case we need.
\begin{lemm}
Assume that $(\mu_t)_t$ is a curve of probability measures on $\mathbb{S}^1$
which is differentiable at $0$ with tangent vector 
$v\in T_{\mu_0} \mathcal{P}(\mathbb{S}^1)$, in the sense that
\[\wassd_2(\mu_t, \mu_0+tv) = o(t)\]
(be reminded that $\mu_0+tv:=(\Id+tv)_\#\mu_0$ is the image of
$tv$ by the exponential map at $\mu_0$).

Then for all smooth function $\psi$, we have
\[\frac{\dd}{\dd t} \int \psi \,\dd\mu_t \Big|_{t=0}= \int \psi' v \,\dd\mu_0.\]
\end{lemm}
Note that this lemma is mostly relevant in the case when $\mu_0$ is
regular in the sense of Gigli (i.e., in dimension $1$, atomless) since then all
``tangent vectors'' at $\mu_0$ are indeed represented by a vector field
$v\in L^2(\mu_0)$.
In a more general manifold, the same result would hold with $\psi$
a compactly supported function, and $\nabla \psi$ instead of $\psi'$.

\begin{proof}
First, we observe that denoting by $\Pi_t$ an optimal transport plan
from $\mu_t$ to $\mu_0+tv$ we have
\begin{align*}
\big|\smallint \psi \,\dd\mu_t -\smallint \psi \,\dd(\mu_0+tv)\big|
  &= \big| \smallint\psi(x) \,\dd\Pi_t(x,y) - \smallint \psi(y) \,\dd\Pi_t(x,y) \big| \\
  &\le \smallint |\psi(x)-\psi(y)| \,\dd\Pi(x,y) \\
  &\le \sup(\psi') \smallint d(x,y) \,\dd\Pi(x,y) \\
  &\le \sup(\psi') \big(\smallint d(x,y)^2 \,\dd\Pi(x,y) \big)^{\frac12}
    \big(\smallint 1 \,\dd\Pi(x,y) \big)^{\frac12} \\
  &= \sup(\psi') \wassd_2(\mu_t,\mu_0+tv)\\
  &= o(t)
\end{align*}
so that we can use $\mu_0+tv$ to estimate the derivative of the integral of $\psi$.

Next, we have
\begin{align*}
\int \psi \,\dd(\mu_0+tv) - \int \psi\,\dd\mu_0 
  &= \int \big(\psi\circ(\Id+tv) -\psi\big) \,\dd\mu_0 \\
  &= \int \big(\psi(x+tv(x))-\psi(x)\big) \,\dd\mu_0 \\
  &= \int \big(t\psi'(x)v(x) + \frac12\psi''(c_x)t^2 v(x)^2 \Big) \,\dd\mu_0 \\ 
  &= t \int \psi' v\,\dd\mu_0 + O(t^2).
\end{align*}
Note that the finiteness of $\lVert v\rVert_{L^2(\mu_0)}$ is part of the definition
of the tangent space at $\mu_0$.

Now the claimed equality follows readily from these two estimates.
\end{proof}

\begin{proof}[Proof of Corollary \ref{coro:main}]
Let $v\in \cap_{d\ge2} E_d$ be a non-zero tangent vector at $\lambda$
invariant under all the $D_\lambda\Phi_{d\#}$, and define
$\mu_t=\lambda+tv$. 

By definition of the tangent space at $\lambda$,
$\int v \,\dd\lambda=0$ so that $v$ has a well-defined antiderivative.
Let $\psi_0$ be a smooth approximation of one of its anti-derivatives,
so that $\int\psi'_0v \,\dd\lambda \simeq \int v^2 \,\dd\lambda$
is non-zero.

Then the pointwise continuity equation implies that
\[\frac{\dd}{\dd t} \int \psi_0 \,\dd\mu_t \Big|_{t=0}
  = \int \psi_0' v \,\dd\lambda \neq 0.\]
Moreover, the invariance of $v$ means that the curve  $(\Phi_{d\#}\mu_t)$
is also differentiable at $t=0$, with tangent vector $v$. In consequence, we get
for all smooth function $\psi$:
\[ \frac{\dd}{\dd t} \int \psi \,\dd(\Phi_{d\#}\mu_t) \Big|_{t=0}
  = \int \psi' v \,\dd\lambda = \frac{\dd}{\dd t} \int \psi \,\dd\mu_t \Big|_{t=0}.\]
  
The weak continuity of $(\mu_t)$ is obvious, and the fact
that $\mu_t$ is atomless for almost all $t$ is Proposition
\ref{prop:atomless}.
\end{proof}

%%%%%%%%%%%%%
\subsection*{Acknowledgements} I am indebted to Artur Oscar Lopes for his
numerous questions and comments on the various versions of this paper,
and it is a pleasure to thank him.

I also wish to thank Fr\'ed\'eric Faure, 
\'Etienne Ghys, Nicola Gigli, Antoine Gournay, Nicolas Juillet
and Herv\'e Pajot for 
interesting discussions and their comments on earlier versions of this paper,
and the anonymous referees for their corrections and the
constructive criticism of one of them.

\bibliographystyle{smfalpha}
\bibliography{biblio.bib}

%%%%%%%%%%%%%%%%%%%%%%%%%%%%%%%%%%%%%%%%%%%%%%%%%%%%%%%%%%%%%%%%%%%%%%%%%%%%%%%%
%%%%%%%%%%%%%%%%%%%%%%%%%%%%%%%%%%%%%%%%%%%%%%%%%%%%%%%%%%%%%%%%%%%%%%%%%%%%%%%%
%%%%%%%%%%%%%%%%%%%%%%%%%%%%%%%%%%%%%%%%%%%%%%%%%%%%%%%%%%%%%%%%%%%%%%%%%%%%%%%%
\end{document}